\documentclass{amsart}
\usepackage[latin1]{inputenc}
\usepackage{color}
\usepackage{url}
\usepackage{amssymb}
\usepackage[all]{xy}
\newtheorem{thm}{Theorem}
\newtheorem{cor}[thm]{Corollary}
\newtheorem{lem}[thm]{Lemma}
\newtheorem{prop}[thm]{Proposition}
\theoremstyle{definition}

\theoremstyle{remark}
\newtheorem{rem}[thm]{Remark}
\newtheorem{prob}{Problem}

\numberwithin{equation}{section}

\newcommand{\To}{\longrightarrow}

\begin{document}
\setcounter{tocdepth}{1}

\title[]{The free Banach lattices generated by $\ell_p$ and $c_0$}

\author[A. Avil\'es]{Antonio Avil\'es}
\address{Universidad de Murcia, Departamento de Matem\'{a}ticas, Campus de Espinardo 30100 Murcia, Spain.}
\email{avileslo@um.es}
\author[P. Tradacete]{Pedro Tradacete}
\address{Mathematics Department, Universidad Carlos III de Madrid, E-28911 Legan\'es, Madrid, Spain.}
\email{ptradace@math.uc3m.es}

\author[I. Villanueva]{Ignacio Villanueva}
\address{Universidad Complutense de Madrid, Departamento de Análisis Matemático y Matemática Aplicada, Plaza de Ciencias 3, 28040 Madrid, Spain.}
\email{ignaciov@ucm.es}


\thanks{First author was supported by projects MTM2014-54182-P and MTM2017-86182-P (MINECO,AEI/FEDER,UE) and by 19275/PI/14 (Fundaci\'on S\'eneca). Second author was supported by the MINECO Grants MTM2016-75196-P and MTM2016-76808-P. Third author was supported by MINECO and QUITEMAD+-CM (MTM2014-54240-P), and Comunidad de Madrid (S2013/ICE- 2801)}

\begin{abstract}
We prove that, when $2<p<\infty$, in the free Banach lattice generated by $\ell_p$ (respectively by $c_0$), the absolute values of the canonical basis form an $\ell_r$-sequence, where $\frac{1}{r} = \frac{1}{2} + \frac{1}{p}$ (respectively an $\ell_2$-sequence). In particular, in any Banach lattice, the absolute values of any $\ell_p$ sequence always have an upper $\ell_r$-estimate. Quite surprisingly, this implies that the free Banach lattices generated by the nonseparable $\ell_p(\Gamma)$ for $2<p<\infty$, as well as $c_0(\Gamma)$, are weakly compactly generated whereas this is not the case for $1\leq p\leq 2$. 
\end{abstract}

\maketitle

\section{Introduction}

Free objects play a fundamental role in many areas of mathematics, providing a source of distinguished and natural examples. In this work, we will explore the properties of free Banach lattices generated by the classical $\ell_p$ and $c_0$ spaces.

Recently, the notion of free Banach lattice generated by a Banach space has been introduced \cite{ART}. Namely, given a Banach space $E$, the free Banach lattice generated by $E$, is a Banach lattice $FBL[E]$ such that:
\begin{enumerate}
\item There is a linear isometry $\phi_E$ of~$E$ into~$FBL[E]$, and we write $\phi_E(x) = \delta_x$.
\item For every Banach lattice $X$ and every operator $T:E\rightarrow X$ there exists a unique lattice homomorphism
$\hat T:FBL[E]\rightarrow X$ such that $\|\hat T\|=\|T\|$ and $\hat T\circ \phi_E=T$, i.e. the following diagram commutes:
$$
	\xymatrix{E\ar_{\phi_E}[d]\ar[rr]^T&&X\\
	FBL[E]\ar_{\hat{T}}[urr]&& }
$$
\end{enumerate}
This construction generalizes the notion of free Banach lattice over a set of generators, introduced in \cite{dePW}, and provides a new tool for understanding better the relation between Banach spaces and Banach lattices.

Historically, free vector lattices have been considered much earlier in the literature \cite{Baker, Bleier}. Let us recall that given a set $A$, the free vector lattice over $A$, $FVL(A)$, is a vector lattice containing a set of lattice free generators idempotent with $A$; in other words, there is a set of generators of the size of $A$ in $FVL(A)$ which have no prescribed lattice relation among themselves and $FVL(A)$ is spanned by these generators by means of the lattice operations (see also \cite[Section 3]{dePW}). Recently, B. de Pagter and A. Wickstead in \cite{dePW} introduced the free Banach lattice generated by a set. Essentially, in order to construct free Banach lattices over a set of generators (and in particular, to show they do exist), one needs to find what is the largest possible lattice norm that a free vector lattice can carry. This has been done in \cite[Theorem 4.7]{dePW} and more generally in \cite[Theorem 2.5]{ART}.

To motivate our goal in this paper, let us recall the following fact given in \cite[Theorem 5.4]{ART}: for $p\in [1,2]$, if $(e_n)_{n\in\mathbb N}$ denotes the canonical basis of $\ell_p$, then there is a constant $C>0$ such that for every scalars $(\lambda_i)_{i=1}^m$
\begin{equation}\label{p<2}
C \sum_{i=1}^m |\lambda_i| \leq \Big\|\sum_{i=1}^m \lambda_i |\delta_{e_i}|\Big\|_{FBL[\ell_p]}\leq  \sum_{i=1}^m |\lambda_i|.
\end{equation}
Thus, we have a sequence of vectors $(\delta_{e_i})_{i\in\mathbb N}$ which are equivalent to the $\ell_p$ basis in $FBL[\ell_p]$, while the sequence of their moduli $(|\delta_{e_i}|)_{i\in\mathbb N}$ are equivalent to the $\ell_1$ basis (and in particular not weakly null). This motivates the following fundamental question: how different can be the Banach space structure of a certain basic sequence $(x_i)_{i\in\mathbb N}$ in a Banach lattice from that of its moduli $(|x_i|)_{i\in\mathbb N}$? 

This question is clearly connected to understanding the rigidity that the lattice structure imposes on the isomorphic embeddings of certain Banach spaces into Banach lattices. Consider for example the case of $c_0$: let $(x_n)_{n\in\mathbb N}$ denote a sequence equivalent to the unit vector basis of $c_0$ in a Banach lattice; what can be said about the sequence $(|x_n|)_{n\in\mathbb{N}}$? In this respect, recall the well-known fact that a Banach lattice contains a subspace isomorphic to $c_0$ if and only if it contains a sublattice isomorphic to $c_0$ (cf. \cite[Theorem 4.60]{AB}). However, this result does not give any relevant information about the relation between the linear copy of $c_0$ and the corresponding lattice copy. In fact, since the lattice operations on a Banach lattice are only occasionally weakly continuous, one does not a priori even know whether the sequence $(|x_n|)_{n\in \mathbb N}$ as above is also weakly null. We will later show that this is indeed the case (see Remark \ref{c_0}). 

Moreover, a direct application of \eqref{p<2} yields that, for any set $\Gamma$, if $p\in [1,2]$, then $FBL[\ell_p(\Gamma)]$ contains a subspace isomorphic to $\ell_1(\Gamma)$, and thus, when $\Gamma $ is uncountable, this space is not weakly compactly generated (in fact, it does not even embed as a subspace of a weakly compactly generated space). This provides the first known example of a Banach lattice which is weakly compactly generated as a Banach lattice but not as a Banach space, answering in the negative a question of J. Diestel (we refer to \cite{AGLRT, ART} for more details on this question).

In this paper, we complete this direction of research showing that, surprisingly, the situation when $p>2$ is quite the opposite. First, we will see that if $(e_n)_{n\in\mathbb N}$ is the canonical basis of $\ell_p$, then for $p>2$ the sequence $(|\delta_{e_n}|)_{n\in\mathbb N}$ in $FBL[\ell_p]$ spans a subspace isomorphic to $\ell_r$ with $\frac1r=\frac12+\frac1p$ (Theorem \ref{p:p>2}). Similarly, in $FBL[c_0]$, the sequence $(|\delta_{e_n}|)_{n\in\mathbb N}$ can be seen to be equivalent to the $\ell_2$ basis. These facts will later be used to show that, unlike when $p\leq2$, the spaces $FBL[\ell_p(\Gamma)]$ for $p>2$ and $FBL[c_0(\Gamma)]$ are weakly compactly generated, independently of the size of $\Gamma$ (Corollary \ref{wcg}). 

By the universal character of free Banach lattices, there is a counterpart of these results for general Banach lattices. If $(x_n)_{n\in\mathbb{N}}$ is a sequence of vectors in a Banach lattice that is equivalent to the canonical basis of $\ell_p$ for $p>2$ (respectivel of $c_0)$, then the sequence of absolute values $(|x_n|)_{n\in\mathbb{N}}$ has an upper $\ell_r$-estimate (respectively an upper $\ell_2$-estimate). And no better estimates can be expected. Also, if a subspace $Z$ of a Banach lattice is isomorphic, as a Banach space to $c_0(\Gamma)$ or $\ell_p(\Gamma)$ for $p>2$, then the Banach lattice generated by $Z$ is a weakly compactly generated Banach space.

We refer the reader to the monographs \cite{AB, LT2, MN} for unexplained terminology and background on Banach lattice theory.

\section{Preliminary results}

The existence of $FBL[E]$ is not evident per se. Let us begin with some observations to motivate the explicit expression for the norm of $FBL[E]$ given in \cite{ART}, which in particular shows its existence.  

Let $H[E]\subset\mathbb R^{\mathbb R^E}$ be the linear space of all positively homogeneous functions $f:E^\ast \to \mathbb{R}$. The free vector lattice $FVL(E)$ can be identified with the vector sublattice of $H[E]$ generated by the evaluations $\{\delta_x: x\in E\}$, where $\delta_x(f)=f(x)$ for every $f\in \mathbb R^E$. The norm of $FBL[E]$ must be the largest possible lattice norm that we can define on this space. In particular, given arbitrary $(x_k^*)_{k=1}^n\subset E^*$ we can define an operator $T: E \rightarrow \ell_1^n$ by the expression 
$$
T(x)=(x_k^*(x))_{k=1}^n
$$
for $x\in E$. It is easy to check that the lattice homomorphism $\hat{T}:FBL[E]\rightarrow \ell_1^n$ extending $T$, is necessarily given by
$$
T(f)=(f(x_k^*))_{k=1}^n
$$
for $f\in FVL(E)$. Hence, the norm of $FBL[E]$, provided it exists, for $f\in FVL(E)$, must satify the inequality
$$
\|\hat T f\|_{\ell_1^n}\leq \|T\| \|f\|_{FBL[E]}.
$$
Therefore, we have
$$
\|f\|_{FBL[E]}\geq \frac{\|\hat T f\|_{\ell_1^n}}{\|T\|}=\frac{\sum_{k=1}^n |f(x_k^*)|}{\sup_{x\in B_E}\sum_{k=1}^n|x_k^*(x)|}.
$$

This motivates an explicit expression for the norm of $FBL[E]$ as follows: for any $f\in H[E]$, define the norm
$$
	\|f\|_{FBL[E]} :=
	\sup\left\{\sum_{k=1}^n |f(x_k^\ast)| : \, n\in\mathbb N, \, x_1^*,\dots,x_n^*\in E^*, \,  \sup_{x\in B_E} \sum_{k=1}^n |x_k^\ast(x)|\leq 1\right\}.
$$
It was shown in \cite[Theorem 2.4]{ART} that the Banach lattice $FBL[E]$ coincides with the closed (with respect to the above norm) sublattice of $H[E]$ generated by $\{\delta_x:x\in E\}$, and the linear isometry $\phi_E:E\rightarrow FBL[E]$ is given by $\phi_E(x)=\delta_x$.
\medskip

The problem that we will study for the canonical basic sequences of $\ell_p$ and $c_0$ could as well be considered for more general sequences of vectors. Let us make some general remarks and point out the problem.

\begin{prop} Let $(e_n)_{n\in\mathbb N}$ be a sequence of vectors in a Banach space $E$, and let $(|\delta_{e_n}|)_{n\in\mathbb N}$  be the sequence of its absolute values in $FBL[E]$.
	\begin{enumerate}
		\item If $(e_n)_{n\in\mathbb N}$ is a basic sequence, then so is $(|\delta_{e_n}|)_{n\in\mathbb N}$.
		\item If $(e_n)_{n\in\mathbb N}$ is unconditional, then so is $(|\delta_{e_n}|)_{n\in\mathbb N}$.
		\item If $(e_n)_{n\in\mathbb N}$ is symmetric, then so is $(|\delta_{e_n}|)_{n\in\mathbb N}$, and therefore $(|\delta_{e_n}|)_{n\in\mathbb N}$ is either an $\ell_1$-sequence or it is weakly null.
	\end{enumerate}
\end{prop}

\begin{proof}
	Suppose without loss of generality that $E =\overline{span}\{e_n\}$ and $A\subset\mathbb{N}$ and we have a projection $T_A:E\To E$ such that $T_A(e_n) = e_n$ if $n\in A$ and $T_A(e_n) = 0$ if $n\not\in A$. Then, we can extend this projection using the universal property of the free Banach lattice to a Banach lattice homomorphism $\hat{T}_A:FBL[E]\To FBL[E]$ that satisfies $\hat{T}_A(|\delta_{e_n}|) = |\delta_{e_n}|$ if $n\in A$ and $\hat{T}_A(|\delta_{e_n}|) = 0$ if $n\not\in A$. The use of this kind of operators gives statements (1) and (2). For (3), if the basis is symmetric, for each permutation $\sigma$ of the natural numbers we have an isomorphism $T_\sigma:E\To E$ such that $Te_n = e_{\sigma(n)}$. Extending this operator to a lattice homomorphism in $FBL[E]$ we get statement (3). The last part follows from Rosenthal's $\ell_1$-lemma.
\end{proof}

\begin{prob}
Given a bounded sequence $(e_n)_{n\in\mathbb N}$ in a Banach space $E$, characterize when $(|\delta_{e_n}|)_{n\in\mathbb N}$ is weakly null in $FBL[E]$ and when it is equivalent to the basis of $\ell_1$.
\end{prob}

Given the universal property of free Banach lattices, this last question is as much as to ask when is the sequence of absolute values $(|v_n|)_{n\in\mathbb N}$ weakly null for \emph{any} sequence $(v_n)_{n\in\mathbb N}$ equivalent to $(e_n)_{n\in\mathbb N}$ in \emph{any} Banach lattice, and when is the sequence of absolute values $(|v_n|)_{n\in\mathbb N}$ an $\ell_1$-sequence for \emph{some} sequence $(v_n)_{n\in\mathbb N}$ equivalent to $(e_n)_{n\in\mathbb N}$ in \emph{some} Banach lattice.

Of course, one can wonder what happens with other particular lattice expressions, instead of the absolute value. For example, we have the following simple fact about the positive and negative part:

\begin{prop}
	Let $(e_n)_{n\in\mathbb N}$ be a sequence of vectors in a Banach space $E$. The sequences $((\delta_{e_n})_+)_{n\in\mathbb N}$ and $((\delta_{e_n})_-)_{n\in\mathbb N}$ are 1- equivalent in $FBL[E]$.
\end{prop}

\begin{proof}
	Let $T:E\rightarrow FBL[E]$ be given by $T=-\phi$, in other words $T(x)=-\delta_x$ for $x\in E$. By the free property of $FBL[E]$ there is only one lattice homomorphism $\hat T:FBL[E]\rightarrow FBL[E]$ extending $T$ with $\|\hat T\|=\|T\|=1$. Since
	$$
	\hat T( (\delta_{e_n})_+)=(\delta_{e_n})_-, \quad\quad \text{and} \quad\quad \hat T( (\delta_{e_n})_-)=(\delta_{e_n})_+,
	$$
	the conclusion follows.
\end{proof}

\section{$FBL[c_0]$ and $FBL[\ell_p]$ for $p\geq 2$}\label{FBL[ell_p]}

The aim of this section is to determine the behavior of the sequence $(|\delta_{e_n}|)_{n\in\mathbb N}$ in $FBL[c_0]$ and $FBL[\ell_p]$ for $p\geq2$, where $(e_n)_{n\in\mathbb N}$ is the unit vector basis of the corresponding space.

\begin{lem}\label{l:krivinep}
Let $p\geq2$ and $X$ be a Banach lattice. Let $T:\ell_p\rightarrow X$ be a bounded linear operator and let $y_n=Te_n$ where $(e_n)_{n\in\mathbb N}$ denotes the canonical basis of $\ell_p$. For any scalars $(\lambda_i)_{i=1}^m$ we have
$$
\Big\|\sum_{i=1}^m \lambda_i |y_i|\Big\|\leq \|T\| K_G \Big(\sum_{i=1}^m |\lambda_i|^r\Big)^{\frac1r},
$$
where $\frac1r=\frac12+\frac1p$ and $K_G$ denotes Grothendieck's constant.
\end{lem}

\begin{proof}
For any scalars $(\lambda_i)_{i=1}^m$, using the fact that $\frac1r=\frac12+\frac1p$ and Cauchy-Schwarz inequality, we get
$$
\Big\|\sum_{i=1}^m \lambda_i |y_i|\Big\|\leq \Big\|\sum_{i=1}^m |\lambda_i|^{\frac{r}{2}} |\lambda_i|^{\frac{r}{p}} |y_i|\Big\|\leq\Big(\sum_{i=1}^m |\lambda_i|^r\Big)^{\frac12} \Big\|\Big(\sum_{i=1}^m (|\lambda_i|^{\frac{r}{p}}|y_i|)^2\Big)^{1/2}\Big\|.
$$

By Krivine's analogue of Grothendieck's inequality for Banach lattices \cite[Theorem 1.f.14]{LT2}, given any choice of vectors $(x_i)_{i=1}^m\subset \ell_p$, we have that
$$
\Big\|\Big(\sum_{i=1}^m |Tx_i|^2\Big)^{1/2}\Big\|\leq K_G\|T\| \Big\|\Big(\sum_{i=1}^m |x_i|^2\Big)^{1/2}\Big\|.
$$
Therefore, choosing $x_i=|\lambda_i|^{\frac{r}{p}} e_i$ in $\ell_p$, it follows that
$$
\Big\|\Big(\sum_{i=1}^m (|\lambda_i|^{\frac{r}{p}}|y_i|)^2\Big)^{1/2}\Big\|\leq K_G\|T\| \Big\|\Big(\sum_{i=1}^m (|\lambda_i|^{\frac{r}{p}} |e_i|)^2\Big)^{1/2}\Big\|\leq K_G\|T\| \Big(\sum_{i=1}^m |\lambda_i|^r\Big)^{\frac{1}{p}}.
$$
Thus, we get
$$
\Big\|\sum_{i=1}^m \lambda_i |y_i|\Big\|\leq K_G\|T\|\Big(\sum_{i=1}^m |\lambda_i|^r\Big)^{\frac12} \Big(\sum_{i=1}^m |\lambda_i|^r\Big)^{\frac1p}= K_G\|T\| \Big(\sum_{i=1}^m |\lambda_i|^r\Big)^{\frac1r}.
$$
\end{proof}

\begin{rem}
The above proof in the case $p=\infty$ corresponds to the same result for the space $c_0$ with $r=2$. That is, if $T:c_0\rightarrow X$ is a bounded linear operator, then, for any scalars $(\lambda_i)_{i=1}^m$, we have
$$
\Big\|\sum_{i=1}^m \lambda_i |Te_i|\Big\|\leq \|T\| K_G \Big(\sum_{i=1}^m |\lambda_i|^2\Big)^{\frac12}.
$$
\end{rem}

We will see next that the previous estimate is sharp. More precisely, we have the following:

\begin{thm}\label{p:p>2}
Let $p>2$ and take $r$ such that $\frac1r=\frac12+\frac1p$. If $(e_n)_{n\in\mathbb N}$ is the canonical basis of $\ell_p$, then for every scalars $(\lambda_i)_{i=1}^m$
$$
\Big(\sum_{i=1}^m |\lambda_i|^r\Big)^{\frac1r}\leq \Big\|\sum_{i=1}^m \lambda_i |\delta_{e_i}|\Big\|_{FBL[\ell_p]}\leq K_G \Big(\sum_{i=1}^m |\lambda_i|^r\Big)^{\frac1r}.
$$
\end{thm}

\begin{proof}
Since $(\delta_{e_n})_{n\in\mathbb N}$ span an isometric copy of $\ell_p$ in $FBL[\ell_p]$, the right-hand side inequality follows from Lemma \ref{l:krivinep}.

For the converse inequality, given  $(\lambda_i)_{i=1}^m$, let $b_i=|\lambda_i|^{r-1}/(\sum_{j=1}^m |\lambda_j|^r)^{1-\frac1r}$. Without loss of generality assume $m=2^k$ for some $k\in\mathbb N$, and let $W=(w_{ij})_{i,j=1}^{2^k}$ be a Walsh matrix, that is $w_{ij}\in\{+1,-1\}$ such that the rows (and columns) are orthogonal. For $j\leq m$, let $x_j^*\in \ell_{p}^*$ be given by
$$
x_j^*(e_i)=\left\{
\begin{array}{ccc}
 \frac{b_i}{m}w_{ij} &   & \text{for } i\leq m, \\
  &   &   \\
0  &   & \text{otherwise.}
\end{array}
\right.
$$

We claim that $\sup_{x\in B_{\ell_p}} \sum_{j=1}^m |x_j^\ast(x)|\leq 1$.

Indeed, first note that by the orthogonality of the columns in $W$ it follows that for every choice of signs $(\sigma_j)\in\{+1,-1\}^m$
\begin{eqnarray*}
\sum_{i=1}^m\Big|\sum_{j=1}^m w_{ij}\sigma_j\Big|^2&=&\sum_{i=1}^m\Big(\sum_{j=1}^m w_{ij}\sigma_j\Big)\Big(\sum_{l=1}^m w_{il}\sigma_l\Big)\\
&=&\sum_{j,l=1}^m \sigma_j\sigma_l\sum_{i=1}^m w_{ij}w_{il}\\
&=&\sum_{j=1}^m\sigma_j^2 m=m^2.
\end{eqnarray*}
Therefore, using Cauchy-Schwarz inequality, the previous computation and H\"older's inequality we have
\begin{eqnarray*}
\sup_{x\in B_{\ell_p}} \sum_{j=1}^m |x_j^\ast(x)|&=&\sup_{(a_i)\in B_{\ell_p}} \sup_{(\sigma_j)\in\{+1,-1\}^m}  \sum_{j=1}^m \sigma_j x_{j}^\ast( \sum_{i=1}^\infty a_i e_i)\\
&=&\frac1m\sup_{(a_i)\in B_{\ell_p}} \sup_{(\sigma_j)\in\{+1,-1\}^m} \sum_{i=1}^m b_i a_i \sum_{j=1}^m w_{ij}\sigma_j\\
&\leq&\frac1m\sup_{(a_i)\in B_{\ell_p}}\Big(\sum_{i=1}^m |b_i a_i|^2\Big)^{\frac12}\sup_{(\sigma_j)\in\{+1,-1\}^m} \left(\sum_{i=1}^m \left|\sum_{j=1}^m w_{ij}\sigma_j\right|^2\right)^{\frac12}\\
&\leq&\sup_{(a_i)\in B_{\ell_p}}\Big(\sum_{i=1}^m |b_i a_i|^2\Big)^{\frac12}\\
&\leq&\sup_{(a_i)\in B_{\ell_p}}\Big(\sum_{i=1}^m |b_i|^{\frac{r}{r-1}}\Big)^{\frac{r-1}{r}} \Big(\sum_{i=1}^m |a_i|^p\Big)^{\frac1p}\leq1.
\end{eqnarray*}

Hence, it follows that
$$
\Big\|\sum_{i=1}^m \lambda_i |\delta_{e_i}|\Big\|_{FBL[c_0]}\geq \sum_{j=1}^m \sum_{i=1}^m |\lambda_i x_{j}^\ast(e_i)| = \frac1m\sum_{j=1}^m \sum_{i=1}^m |\lambda_i| b_i=\Big(\sum_{i=1}^m |\lambda_i|^r\Big)^{\frac1r}.
$$
\end{proof}

\begin{rem}\label{c_0}
As with Lemma \ref{l:krivinep}, for $p=\infty$, we can take $r=2$, and the above proof yields that if $(e_n)_{n\in\mathbb N}$ is the canonical basis of $c_0$, then for every scalars $(\lambda_i)_{i=1}^m$
$$
\Big(\sum_{i=1}^m \lambda_i^2\Big)^{\frac12}\leq \Big\|\sum_{i=1}^m \lambda_i |\delta_{e_i}|\Big\|_{FBL[c_0]}\leq K_G \Big(\sum_{i=1}^m \lambda_i^2\Big)^{\frac12}.
$$
\end{rem}

\section{When is $FBL[E]$ a WCG space?}

Let us recall a Banach space $E$ is weakly compactly generated (WCG) if there is some weakly compact set $K\subset E$ such that $\overline{span}(K)=E$. We know that when $E$ is isomorphic to a Banach lattice, then it follows from \cite[Corollary 2.6]{ART} that $E$ is WCG if so is $FBL[E]$. However, the converse does not hold in general: we know $FBL[\ell_p(\Gamma)]$ for $p\in[1,2]$ and uncountable $\Gamma$ are not WCG \cite[Corollary 5.5]{ART}. We will see next that the situation when $p>2$ is quite the opposite.\\

\begin{thm}
Let $E$ be a Banach space and $G\subset E$ such that $\overline{span}(G) = E$. Suppose that $(|\delta_{u_m}| )_{m\in\mathbb N}$ is a weakly null sequence in $FBL[E]$ for every infinite sequence of distinct elements $(u_m)_{m\in\mathbb N}\subset G$. Then $FBL[E]$ is a WCG Banach space.
\end{thm}

\begin{proof}
We can suppose that all elements in $G$ have norm at most 1. Enumerate as $P_1(x_1,\ldots,x_{k_1}),P_2(x_1,\ldots,x_{k_2}),P_3(x_1,\ldots,x_{k_3}),\ldots$ all polynomial expressions on finitely many variables using rational linear combinations and the $\wedge$ and $\vee$ operations. For every $n$, let $r_n = \|P_n(\delta_{e_1},\ldots,\delta_{e_{k_n}})\|$ in $FBL[\ell_1]$. By the lifting property of $\ell_1$, whenever $u_1,\ldots,u_{k_n}\in G$, there is an operator $T:\ell_1\To E$ with $\|T\|\leq 1$ and $T(e_i)=u_{k_i}$ for $i\leq n$. By the universal property of the free Banach lattice, $T$ extends to a Banach lattice homomorphism $\hat{T}:FBL[\ell_1]\To FBL[E]$ with $\|\hat{T}\|\leq 1$. A lattice homomorphism commutes with a polynomial expression like $P_n$, so
$$\left\|P_n(\delta_{u_1},\ldots,\delta_{u_{k_n}})\right\| = \left\|\hat{T}(P_n(\delta_{e_1},\ldots,\delta_{e_{k_n}}))\right\|\leq r_n. $$
Let $$S = \left\{\frac{1}{nr_n}P_n(\delta_{u_1},\ldots,\delta_{u_{k_n}}) : n\in\mathbb{N}, r_n\neq 0,  u_1,\ldots,u_{k_n}\in G \right\}.$$

It is clear that $S$ generates $FBL[E]$ as a Banach space. We claim that $S$ is relatively weakly compact. By the Eberlein-\v{S}mulian theorem \cite[Theorem 1.6.3]{AlbKal}, it is enough to show that every sequence $\sigma$ of elements of $S$ has a weakly convergent subsequence. Since $$\left\|\frac{1}{nr_n}P_n(\delta_{u_1},\ldots,\delta_{u_{k_n}})\right\| \leq \frac{1}{n},$$
if a sequence $\sigma$ of elements of $S$ contains elements $\frac{1}{nr_n}P_n(\delta_{u_1},\ldots,\delta_{u_{k_n}})$ for infinitely many different numbers $n$, then it a contains a subsequence that converges to 0 in norm. So we can suppose that our sequence $\sigma$ is made of elements of $S$ in which the number $n$ is fixed. In other words, it is enough to see that for any given polynomial $P(x_1,\ldots,x_k)$ and any elements $\{u_i^m : i=1,\ldots,k,\ m\in \mathbb{N}\}\subset G$, we have that the sequence $$\left\{P\left(\delta_{u_1^m},\ldots,\delta_{u_k^m}\right) : m\in\mathbb{N}\right\}$$
has a weakly convergent subsequence. We can also suppose, by passing to a further subsequence and reordering variables, that we have $q<m$ such that $\{u^m_i : m\in\mathbb{N}\}$ is constant equal to $u_i$ when $i\leq q$, while $u^m_i \neq u^{m'}_i$
whenever $m\neq m'$ and $i>q$. We claim that in that case, the sequence

$$\left\{z_m(P) := P\left(\delta_{u_1},\ldots,\delta_{u_q},\delta_{u_{q+1}^m},\ldots,\delta_{u_k^m}\right) : m\in\mathbb{N}\right\}$$  weakly converges to  $z(P) := P\left(\delta_{u_1},\ldots,\delta_{u_q},0,\ldots, 0\right)$. For this we show, by induction on the complexity of $P$ that $|z_m(P) - z(P)|$ is weakly null. Induction on the complexity means that we first notice that the statement is true for any atomic polynomial of the form $P(x_1,\ldots,x_k) = x_i$, and then we shall prove that if the statement holds for $P$ and $Q$, then it also holds for $P\vee Q$ and any linear combination $\lambda P + \mu Q$ (remember that $P\wedge Q = -(-P\vee - Q)$). The case of the atomic polynomial is given precisely by the hypothesis of the theorem. The case of $P\vee Q$ follows from the following computation, and the fact that a positive sequence bounded above by a weakly null sequence is weakly null:

$$\left|z_m(P\vee Q) - z(P\vee Q)\right| = \left| z_m(P)\vee z_m(Q) - z(P)\vee z(Q)\right| \leq$$ $$\left| z_m(P) - z(P)\right| \vee \left|z_m(Q) - z(Q)\right| \leq \left| z_m(P) - z(P)\right| + \left|z_m(Q) - z(Q)\right.|$$

The case of the linear combination is similar, using just triangle inequality.
\end{proof}

This last theorem together with the results of Section \ref{FBL[ell_p]} yield the following counterpart to \cite[Corollary 5.5]{ART}:

\begin{cor}\label{wcg}
For any set $\Gamma$ and $p>2$ the Banach lattices $FBL[\ell_p(\Gamma)]$ and $FBL[c_0(\Gamma)]$ are WCG.
\end{cor}

Given $1< p<\infty$, let us say that a basis $(x_\gamma)_{\gamma\in \Gamma}$ of a Banach space satisfies an upper (respectively, lower) $p$-estimate if there is $C>0$ such that for any scalars $(a_\gamma)_{\gamma\in\Gamma}$ we have
$$
\Big\|\sum_{\gamma} a_\gamma x_\gamma\Big\|\leq C \| (a_\gamma) \|_{\ell_p(\Gamma)},
$$
(respectively, 
$$
\Big\|\sum_{\gamma} a_\gamma x_\gamma\Big\|\geq C \| (a_\gamma) \|_{\ell_p(\Gamma)}.)
$$
It is immediate from the Corollary that if a Banach space $E$ has a (long) basis which satisfies an upper $p$-estimate for some $p>2$, then $FBL[E]$ is also WCG. On the other hand, if the long basis of $E$ satisfies a lower $2$-estimate, then $FBL[E]$ contains $\ell_1(\Gamma)$ so it does not embed into a WCG Banach space.
\medskip

In particular, the above shows that for any uncountable set $\Gamma$, if $p\leq 2$ and $q>2$, then $FBL[\ell_p(\Gamma)]$ and $FBL[\ell_q(\Gamma)]$ are not isomorphic even as Banach spaces. It would be interesting to find a direct argument to distinguish the spaces $FBL[\ell_r]$ for different values of $r$. In this direction we have the following:

\begin{thm}\label{pqdif}
Let $p,q\in[1,\infty)$ with $p\neq q$. There is no lattice isomorphism between $FBL[\ell_p]$ and $FBL[\ell_q]$.
\end{thm}

We will need several facts for the proof which might be of independent interest.

\begin{lem}\label{latticehomo}
Given a Banach space $E$, $\varphi:FBL[E]\rightarrow \mathbb R$ is a linear lattice homomorphism if and only if there is $x^*\in E^*$ such that for every $f\in FBL[E]$
$$
\varphi(f)=f(x^*).
$$ 
\end{lem}
\begin{proof}
\cite[Corollary 2.6]{ART}.
\end{proof}

\begin{lem}\label{projection}
If $E$ is a Banach lattice, then there is a lattice homomorphism projection $P:FBL[E]\rightarrow E$.
\end{lem}
\begin{proof}
\cite[Corollary 2.5]{ART}.
\end{proof}

\begin{lem}\label{compact}
Let $p,q\in[1,\infty)$ with $p>q$. Every lattice homomorphism $T:FBL[\ell_p]\rightarrow \ell_q$ is compact.
\end{lem}

\begin{proof}
Let $T:FBL[\ell_p]\rightarrow \ell_q$ be a lattice homomorphism. For $i\in \mathbb N$, let $e_i^*\in \ell_q^*$ denote the biorthogonal functionals corresponding to the unit vector basis of $\ell_q$. That is, for scalars $(a_j)_{j\in\mathbb N}$ 
$$
e_i^*(\sum_{j\in\mathbb N} a_j e_j)=a_i.
$$
Note that for every $i\in\mathbb N$, $e_i^*\circ T:FBL[\ell_p]\rightarrow \mathbb R$ is a lattice homomorphism. Hence, by Lemma \ref{latticehomo}, there exist $x_i^*\in \ell_p^*$ such that for every $f\in FBL[\ell_p]$
$$
e_i^*(T(f))=f(x_i^*).
$$ 
Now, for $n\in\mathbb N$, let $F_n:FBL[\ell_p]\rightarrow \ell_q$ be the finite rank operator given by
$$
F_n (f)=\sum_{i=0}^n f(x_i^*)e_i.
$$
Note that for $f\in FBL[\ell_p]$, we have that 
$$
[T-F_n](f)=\sum_{i>n} f(x_i^*)e_i.
$$ 
Therefore, $T-F_n$ is a lattice homomorphism. Also note that by Pitt's theorem \cite[Proposition 2.c.3]{LT1}, $T\circ \phi_{\ell_p}:\ell_p\rightarrow \ell_q$ is compact. Hence, 
$$
\|(T-F_n)\circ \phi_{\ell_p}\|=\|T\circ \phi_{\ell_p}-F_n\circ \phi_{\ell_p}\|\underset{n\rightarrow \infty}\longrightarrow0.
$$
By construction of $FBL[\ell_p]$, it follows that $\|T-F_n\|=\|(T-F_n)\circ \phi_{\ell_p}\|\rightarrow0$. Thus, $T$ is compact as claimed.
\end{proof}

\begin{proof}[Proof of Theorem \ref{pqdif}]
Without loss of generality suppose that $p>q$ and $$T:FBL[\ell_p]\rightarrow FBL[\ell_q]$$ is a lattice isomorphism. Let $$P:FBL[\ell_q]\rightarrow \ell_q$$ be the lattice homomorphism projection given by Lemma \ref{projection}, and consider the lattice homomorphism $$PT:FBL[\ell_p]\rightarrow \ell_q,$$ which is in particular invertible on a subspace isomorphic to $\ell_q$. This is a contradiction with Lemma \ref{compact}.
\end{proof}

\end{document}